\documentclass[11pt]{article}
\usepackage[top = 1in, bottom = 1in, left = 1in, right = 1in]{geometry}
\usepackage{amsthm, amsfonts,amsmath,amssymb,latexsym}
\usepackage[colorlinks]{hyperref}
\AtBeginDocument{
	\hypersetup{
		colorlinks=true,
 		linkcolor=RoyalBlue,
  		citecolor=RedOrange,
	 }
}
\usepackage{graphicx}
\usepackage[dvipsnames]{xcolor}
\usepackage{pgf,tikz}
\usepackage{tkz-graph}
\usetikzlibrary{decorations.markings}
\usepackage{enumitem}
\usepackage{verbatim}

\newtheorem{theorem}{Theorem}[section]

\newtheorem{lemma}[theorem]{Lemma}
\newtheorem{definition}[theorem]{Definition}
\newtheorem{corollary}[theorem]{Corollary}

\newtheorem{remark}[theorem]{Remark}

\theoremstyle{remark}
\newtheorem*{note}{\bf Note}

\def\pr{\textup{P\/}}
\def\ex{\textup{E\/}}
\def\mean{\textup{E\/}}

\def\eps{\varepsilon}

\def\ga{\gamma}
\def\part{\partial}

\newcommand{\beq}{\begin{equation}}
\newcommand{\eeq}{\end{equation}}

\newcommand{\floor}[1]{\lfloor #1 \rfloor}

\newcommand{\lp}{\left(}

\newcommand{\rp}{\right)}

\newcommand{\din}[1]{D_{\text{in}}(#1)}
\newcommand{\dout}[1]{D_{\text{out}}(#1)}
\newcommand{\indx}{d_i(x)}
\newcommand{\outdx}{d_o(x)}
\newcommand{\gnk}{G_{n,k}}
\newcommand{\dgnk}{{\overrightarrow G}_{n,k}}
\newcommand{\diam}{\textup{diam}}

\renewcommand{\emptyset}{\varnothing}

\numberwithin{equation}{section}
\title{On connectivity, conductance and bootstrap percolation for a random k-out, age-biased
graph}
\date{}
\author{H\"{u}seyin Acan \thanks{Supported by National Science Foundation Fellowship (Award No.~1502650).}\\ Department of Mathematics\\ Rutgers University\\ Piscataway, NJ 08873\\
\texttt{huseyin.acan@rutgers.edu}
\and
Boris Pittel\\
Department of Mathematics\\ The Ohio State University\\ Columbus, OH 43210\\ \texttt{bgp@math.osu.edu}
}

\begin{document}
\maketitle

\begin{abstract}
A uniform attachment graph (with parameter $k$), denoted $G_{n,k}$ in the paper, is a random graph on the vertex set $[n]$, where each vertex $v$ makes $k$ selections from $[v-1]$ uniformly and independently, and these selections determine the edge set. We study several aspects of this graph. Our motivation comes from two similarly constructed, well-studied random graphs: $k$-out graphs and preferential attachment graphs. In this paper, we find the asymptotic distribution of its minimum degree and connectivity, and study the expansion properties of $G_{n,k}$ to show that the conductance of $G_{n,k}$ is of order $(\log n)^{-1}$. 
We also study the bootstrap percolation on $G_{n,k}$, where, each vertex is either initially infected with probability $p$, independently of others, or gets infected later as a result of having $r$ infected neighbors at some point. We show that, for $2\le r\le k-1$, if $p\ll (\log n)^{-r/(r-1)}$, then, with probability approaching 1, the process ends before all vertices get infected. On the other hand, if $p\ge \omega(\log n)^{-r/(r-1)}$, where $\omega$ is a certain very slowly growing function, then all the vertices get infected with probability approaching 1.

\end{abstract}

\section{Introduction}
We study a dynamic random graph model, which is called a \emph{uniform attachment graph} in~\cite{MJKS}. In this model, new vertices are added one at a time to the graph, and each time a vertex is added, it selects $k$ neighbors uniformly and independently (with repetition) from the already present vertices, where $k$ is a parameter of the model. This gives rise to a growing directed multigraph with all out-degrees $k$. Our interest is in the simple undirected graph obtained from this graph by removing the edge orientations and multiple edges. The case $k=1$ gives a random recursive tree, which has been studied extensively.

By construction, uniform attachment graphs are similar to two other families of graphs which have been well studied: preferential attachment graphs and uniform $k$-out graphs. In both graphs, each vertex chooses $k$ neighbors. In the former, the choices heavily depend on the degrees of the previous vertices whereas in the latter the choices are made uniformly from the whole vertex set. As a result, in preferential attachment graphs, there are few vertices of very large degrees (hubs) and many vertices of small degrees, whereas in $k$-out graphs the degrees are largely close to each other, and their joint distribution is symmetric. In a sense, the graphs we consider in this paper lie in between these two families: vertex $v_i$ chooses its out-neighbors from the vertices $\{v_1,\dots,v_{i-1}\}$ as in the preferential attachment graphs, but the choices it makes are uniform and hence do not depend on the degrees of the potential out-neighbors. The in-degrees of the older vertices tend to be larger, but this bias is weaker than the bias in the preferential attachment graphs.

Recently Magner, Janson, Kollias and Szpankowski~\cite{MJKS} studied the symmetry in uniform attachment graphs and showed that the graph, after $n$ vertices are added and as $n\to \infty$, is symmetric (i.e.\ it has a nontirivial automorphism) with high probability (whp) for $k=1$, and with probability bounded away from 0 for $k=2$. They also conjectured that it is asymmetric for $k\ge 3$. More recently, Frieze, P\'erez-Gim\'enez, Pra\l{}at and B. Reiniger~\cite{FGPR} studied the hamiltonicity and perfect matchings in uniform attachment graphs;  they proved that whp the graph has a perfect matching for $k\ge 159$ and it is Hamiltonian for $k\ge 3214$.

In this paper, we study several other aspects of the uniform attachment graph $G_{n,k}$.
We begin with the limiting distribution of the minimum degree. We show that whp the minimum degree is either $k-1$ or $k$, and  that the number of vertices with degree $k-1$ converges in distribution to a Poisson random variable with parameter $(k-1)/2$. We demonstrate that whp the connectivity of the graph is equal to its minimum degree so that the connectivity distribution
is asymptotically supported by $\{k-1,k\}$ as well. (A similar closeness between connectivity and minimum degree is on
display in a classic result proved by Bollob\'as and Thomason \cite{BolTho}: for every deterministic $k$, which may depend
arbitrarily on $n$, whp the moments when the minimum degree and the connectivity of the Erd\"os-R\'enyi evolving random graph weakly exceed $k$, coincide.) 

This result makes it natural to study the expansion properties of $G_{n,k}$. We show  
that whp $G_{n,k}$ is a vertex-expander with a positive  limiting, $k$-dependent, expansion rate. We use this claim to show 
 that whp $G_{n,k}$ is an edge-expander as well, with the expansion rate (i.e.\ {\it conductance\/} $\Phi(G_{n,k}))$
of order $\log^{-1} n$ at least. This implies that whp the total variation distance between the uniform random walk on $G_{n,k}$ and the stationary distribution decays quasi-geometrically, at the rate $\exp(-\Theta(\log^{-2}n))$. It brings to mind the work of Fountoulakis and Reed \cite{FouRee} and later Benjamini, Kozma and Wormald \cite{BenKozWor} who proved, independently, that the total variation mixing time for the walk on the giant component of $G(n,\pr(\text{edge}=\pi))$ and $G(n,m)$ is whp of order $\Theta(\log^2n)$. 

In the second part we study bootstrap percolation on $G_{n,k}$. Bootstrap percolation is a process which starts with an initial set of ``infected'' vertices. Afterward, at each step, an uninfected vertex with at least $r$ infected neighbors becomes infected and stays infected forever. Here $r$ is a parameter of the process, and it is called the activation threshold. If $r=1$, then all vertices in a connected graph get infected at some point as long as there is at least one infected vertex initially. However, for $r>1$, the final set of infected vertices may strongly depend on the graph and the initial set of infected vertices.

Both extremal and probabilistic aspects of bootstrap percolation have been of keen interest in mathematics and statistical mechanics, and have been studied widely. A natural question is: what can we say about the number of infected vertices at a given time for a given set of initally infected vertices? In particular, which sets of infected vertices spread the infection to all vertices?

In the extremal part, for a given graph, the most typical problem is finding a minimum-cardinality set of initially infected vertices that makes all the vertices infected eventually. For the probabilistic angle, one usually starts with a random set of initially infected vertices in a random graph or a deterministic graph and analyzes the distribution of the number of infected vertices at the terminal state of the process.
Given a graph, a popular choice for the random configuration of the infected vertices is the outcome of the following experiment: 
for each of the vertices flip a coin with heads probability $p=p(n)$, and declare the vertex infected if the coin turns up heads.

The problem of finding a threshold value of the initial infection probability $p$ for the property that ``all the vertices become infected eventually'' has been studied for many graphs. 
For example, Balogh and Pittel~\cite{BP} studied this problem for random regular graphs, and Janson, \L uczak, Turova and Vallier~\cite{JLTV}  for $G(n,\pi)$ in the full  range of $\pi:=\pr(\text{edge})$. 
In a more general, non-homogeneous setting, in~\cite{Ami1} and~\cite{Ami2}, Amini studied the percolation for undirected and
directed random graphs with given degree sequences, subject to some restrictions.

The problem has also been studied for scale-free random graphs, where the degrees follow a power law distribution. A well known class of such graphs is the preferential attachment model, which was introduced by Barab\'asi and Albert~\cite{BA} and defined and studied rigorously by Bollob\'as, Riordan, Spencer and Tusn\'ady~\cite{BRST}. Abdullah and Fountoulakis~\cite{AbF} gave a detailed analysis of the bootstrap percolation on generalized preferential attachment graphs. In particular, their results yield immediately that $n^{-1/2}$ is a threshold for the spread of infection to all vertices for the basic Barab\'asi-Albert model.  Similar results were previously also established by Ebrahimi, Gao, Ghasemiesfeh and Schoenenbeck~\cite{EGGS}. They proved, for the Barab\'asi-Albert model, that (i) if $p$ is much smaller $n^{-1/2}$, then the final set of infected vertices is not the whole vertex set and (ii) if $p$ is much bigger than $n^{-1/2}\log n$, then whp the whole set is infected eventually in $O(\log n)$ steps.

Another random graph model with power law distribution was introduced by Chung and Lu~\cite{CL}, which is known as the Chung-Lu model. In this random graph model, vertices $i$ and $j$ are joined by an edge with probability $w(i)w(j)/W$, where $w\,:\,V\rightarrow \mathbb{R}^+$ is a weight function and $W=\sum_{k} w(k)$. Accordingly, assuming $\max_j \{w(i)w(j)\}\le W$, the expected degree of vertex $i$ is $w_i$. With a suitable choice of weight sequence $w$, one gets a power law degree distribution.
For these graphs, Amini and Fountoulakis~\cite{AmF} found a threshold value $q=q(n)$ such that (i) there is no evolution when $p\ll q$ and (ii) a linear size infection occurs when $p\gg q$.
Improving this result, Amini, Fountoulakis and Panagiotou~\cite{AFP}, proved a law of large numbers for the size of the final infection when $p\gg q$. 

In this paper, we add the uniform attachment graph to this list. 
We show that (i) if each vertex is initially infected (independently of each other) with probability $p\ge \omega(\log n)^{-r/(r-1)}$, where $\omega$ is growing fairly slowly, then whp all the vertices become infected eventually, and (ii) if each vertex is infected initially with probability $p\ll(\log n)^{-r/(r-1)}$, then whp the process ends up with some vertices remaining uninfected.

In Section~\ref{sec:results} we introduce our notation and fully state our main results for the uniform attachment graph $G_{n,k}$. In Section~\ref{sec:mindegree}, we study the minimum degree and the connectivity of $G_{n,k}$.  In Section~\ref{sec:expander}, we show that $G_{n,k}$ is a vertex-expander, and an edge-expander, and apply the latter property to determine the likely rate
of convergence of the uniform walk distribution to the stationary one. In Section~\ref{sec:BP}, we prove the two-sided
estimate of the threshold for the bootstrap percolation.

\section{Main results}\label{sec:results}

\noindent \textbf{Notation.} We denote by $[n]$ the set of first $n$ positive integers and by $[m+1,n]$ the set of integers from $m+1$ to $n$. We use the standard notation $O(\cdot), \Theta(\cdot)$, etc.\ to describe the growth rates of functions of interest. We use $f(n)\le_b g(n)$ instead of $f(n)=O(g(n))$ when $g(n)$ is a bulky expression. When $f(n)/g(n)\to 1$ as $n\to \infty$, we write $f(n)\sim g(n)$. We say that an event $E_n$  occurs with high probability (whp), if $\lim_{n\to\infty} \pr(E_n)=1$.

We now introduce the uniform attachment graph $G_{n,k}$, which has the vertex set $[n]$.
Each vertex $u\in [2,n]$ makes $k$ independent selections uniformly from $[u-1]$ and for each selection $v$, the edge $uv$ is added to the graph.
In the end, multiple edges between two vertices are reduced to a single edge. It can also be viewed as the snapshot after $n$ steps of the following graph process $\{G_{t,k}\}_{t\ge1}$. The graph $G_{1,k}$ is the unique edgeless graph with a single vertex 1.
For $t\ge2$, $G_{t,k}$ is obtained from $G_{t-1,k}$ by adding the vertex $t$, and $k$ edges from $t$ to the previous vertices, where the other end of each edge is chosen uniformly and independently from $[t-1]$. (Multiple edges are reduced to a single edge.)

Although our main interest is in $G_{n,k}$, which is an undirected graph, we also use the directed version in our proofs. In the directed version $\dgnk$, the edges are oriented toward smaller vertices. In other words, each vertex selects the endpoints of edges emanating from itself.

$G_{n,k}$ is a variant of a well-known $k$-out graph, in which vertex $u$ chooses its neighbors from the whole set $[n]$ instead of $[u-1]$. Hence, although the average degrees are about the same, in $G_{n,k}$ there is a clear bias toward the smaller vertices. On the other hand, this bias is not as strong as the bias in preferential attachment graphs, where vertex $u$ chooses $k$ neighbors from $[u-1]$, not uniformly but proportionally to the degrees of the previous vertices. Whereas whp a preferential attachment graph has maximum degree of polynomial order, whp the maximum degree of $G_{n,k}$ is only $O(\log n)$.

We denote by $\din{v}$ and $\dout{v}$ the in-degree and the out-degree of the vertex $v$, respectively, in $\dgnk$. The degree of a vertex $v$ in $\gnk$, denoted $D(v)$, is the sum of $\din{v}$ and $\dout{v}$.
Note that $\dout{v}\le k$ for each $v\in [n]$, and easy computations show $\mean[\din{v}] \sim k\log(n/v)$ and $\max_v\din{v}\sim k\log n$ whp. Consequently, $\max_v D(v)\sim k\log n$ whp.
We denote by $\delta_n$ the minimum degree of $G_{n,k}$. Our first result gives the limiting distribution of $\delta_n$.

\begin{theorem}\label{thm:mindeg}
Whp, $\delta_n\in \{k-1,k\}$ and
\[
\lim_{n\to\infty}\!\pr(\delta_n =k)=e^{-(k-1)/2}.
\]
\end{theorem}

Next we study the connectivity of $G_{n,k}$.
Given a graph $G$,  let  $\kappa(G)$ denote  the connectivity of $G$, that is, $\kappa(G)$ is the largest $\nu$ such that deletion of any $\nu-1$ vertices does not
separate $G$. Let $\kappa_n=\kappa(G_{n,k})$.

\begin{theorem}\label{thm:connectivity}
For $k\ge 2$,
\[
\pr(\kappa_n=\delta_n)=1-O(n^{-1/k}).
\]
\end{theorem}

Our next result shows that $G_{n,k}$ is an expander.
Given a graph $G=(V,E)$ and a vertex set $X\subseteq V$, let $N(X)$ be the set of all neighbors of $X$
outside of $X$. Since the minimum degree of $G_{n,k}$ is at least $k-1$ whp by Theorem~\ref{thm:mindeg}, for every non-empty vertex set $X$ in $G_{n,k}$, we have $|X|+|N(X)|\ge k$ whp. 
Let $\rho(k)$ be the positive root of
\begin{align*}
&\qquad\qquad\qquad f(\rho)+kg(\rho)=0,\\
&f(\rho):=\log 2-\frac{1}{2}\bigl[\rho\log \rho+(1-\rho)\log(1-\rho)\bigr]\\
&g(\rho):=\frac{1+\rho}{2}\log(1+\rho)-\rho\log \rho-\left(1-\frac{\rho}{2}\right)\log(2-\rho)=0.
\end{align*}
The root $\rho(k)$ uniquely exists for all $k\ge 2$, and $\lim\rho(k)=\rho^*$, which is the unique root of $g(\rho)=0$.  Maple delivers $\rho(2)\approx 0.076$, $\rho(3)\approx 0.114$, $\rho(4)\approx 0.141$, $\rho(5)\approx 0.160$, $\rho(6)\approx 0.173$,
$\rho(7)\approx 0.183$, and $\rho^*\approx 0.252$.

\begin{theorem}\label{expansion} Let $k\ge 2$. For every $\rho<\rho(k)$,  
\begin{equation*}
\lim_{n\to\infty}\!\pr\!\left\{\exists\, X\subseteq [n] : 1\le |X|\le n/2\text{ and }|N(X)|\le\rho |X|\right\}=0.
\end{equation*}
In words, whp every non-empty vertex set $X$ of cardinality $\le n/2$  has at least $\rho |X|$ neighbors outside of $X$, so that
the graph $G_{n,k}$ is a vertex-expander with expansion rate $\rho$, at least.
\end{theorem}

Now the edge-expansion rate (conductance) $\Phi$ of a graph $G(V,E)$ is defined as $\min_{S\subset V}\Phi(S)$, where
\[
\Phi(S):=\frac{|\nabla(S)|\cdot |E|}{d(S)\cdot d(V\!\setminus\! S)},\quad d(T):=\sum_{v\in T}\text{deg}_G(v),
\]
and $\nabla(S)$ is the set of edges connecting $S$ to $V\setminus S$. We will prove existence of a constant $c$ such that
whp $d(S)\le c |S|\bigl(1+\log (n/|S|)\bigr)$ for all non-empty $S\subset V$. Applying Theorem \ref{expansion}, we get that 
there is an absolute constant $\gamma$ such that whp $\Phi(G_{n,k})\ge \gamma \log^{-1} n$.  Let $\{P^t(j\,|\,i)\}_{j\in [n]}$ denote the probability that the uniformly random walk on $[V]$ is in vertex $j$ at time $t$, given the starting state $i\in [n]$, and let $\{\pi(j)=(2|E|)^{-1}\text{deg}_G(j)\}_{j\in [n]}$ denote the stationary distribution of the walk. According to Sinclair--Jerrum
theorem \cite{SinJer} for a general graph $G$, and the likely lower bound for conductance $\Phi(G_{n,k})$, we have

\begin{theorem}\label{walk} 
For the random walk on $G_{n,k}$, there are some constants $\ga_1$ and $\ga_2$ such that
\[
\Bigl|P^t(j\,|\.i)-\pi(j)\Bigr|\le \sqrt{\frac{\pi(j)}{\pi(i)}}\left(1-\frac{\Phi^2(G_{n,k})}{8}\right)^t\le \gamma_1(\log n)^{1/2} \exp\left(-\frac{\gamma_2\,t}{\log^2 n}\right),
\]
the last bound holding with high probability.
\end{theorem}

We turn now to formulation of the results regarding the bootstrap percolation on $G_{n,k}$.

\begin{theorem}\label{thm:percolation1} 
Let $\log^{(s)}$ stand for the $s$-fold composition of $\log$ with itself.
If 
\[
p\ge \omega (\log n)^{-r/(r-1)},
\]
where $\omega=\bigl(3\log^{(3)} n \cdot \log^{(4)}n)^{\frac{r}{r-1}}$, then all the vertices get infected eventually whp.
\end{theorem}

A key element of the argument is the whp existence proof for perfect $\ell$-ary trees rooted at vertices from a slowly growing interval of ``smallish'' vertices, with generations coming from a partition of all other vertices into a slowly growing number of intervals, and leaves from the set of initially infected vertices.

\begin{theorem}\label{thm:percolation2} 
Let $\omega\to\infty$ however slowly. If
\begin{align*}
p\le \omega^{-1}(\log n)^{-r/(r-1)}, 
\end{align*}
then, whp, spread of infection stops before any vertex smaller than $\omega^{1/2}(\log n)^{r/(r-1)}$ is infected.
\end{theorem}
The proof is based on the idea of ``witnesses'', which in this case are rooted directed $r$-ary subgraphs of $G_{n,k}$ with ``leaves''
from the set of initially infected vertices that must be present for a root to be eventually infected, 
cf. Frieze and Pittel \cite{FriPit} and Abdullah and Fountoulakis \cite{AbF}.

We conclude this section with the following standard Chernoff inequalities, which we will use in our proofs in the next sections. (See for instance~\cite[Theorem 2.1 and Theorem 2.8 ]{JLR}.)

\begin{lemma}
Let $X$ be a sum of independent Bernoulli random variables with $\mean[X]=\mu$. Then
\begin{align}
&\pr(X\le \mu-t)  \le e^{-t^2/(2\mu)}  \qquad (t>0), \label{C1} \\
&\pr(|X-\mu| \ge \eps\mu)  \le 2e^{-\eps^2\mu/3} \qquad (0<\eps\le 3/2).		\label{C3}
\end{align}
\end{lemma}


\section{Minimum degree and connectivity of $G_{n,k}$}\label{sec:mindegree}

We prove Theorem~\ref{thm:mindeg} and Theorem~\ref{thm:connectivity} in this section.

\subsection{Minimum degree in $\gnk$}
Recall that $\din{v}$ and $\dout{v}$ denote the in-degree and the out-degree of $v$ in $\dgnk$, respectively, and $D(v)=\din{v}+\dout{v}$.

Since $\din{n}=0$ and $\dout{n}\le k$, we have $D(n)\le k$ and hence the minimum degree of $\gnk$ is at most $k$. The next  two lemmas show that there is no vertex with degree less than $k-1$, whp.  

\begin{lemma}\label{big in-degrees}
There is no vertex in $[\floor{\log n}]$ with in-degree less than $k$ in $\dgnk$.
\end{lemma}

\begin{proof}
Write $m=\floor{\log n}$ and fix $x \in [m]$. 
The in-degree of $x$ satisfies
\[
\din{x} \ge Z_m(x):=\sum_{j\ge m+1} \xi_j(x),
\]
where $\xi_j(x)$ is the indicator of the event that $j$ chooses $x$, i.e., $(j,x)$ is an edge in $\dgnk$. Note that $\xi_j(x)$'s are independent. Also, uniformly for $j>m$,
\[
\pr(\xi_j(x)=1)=1- \left(1- \frac{1}{j-1}\right)^k \sim \frac{k}{j}
\]
and hence 
\[
 \mu:=\mean[Z_m(x)] \sim k\log(n/m)\sim k\log n.
\]
Chernoff bound~\eqref{C1} yields
\[
\pr(Z_m(x)\le \mu/2)=\pr(Z_m(x) \le \mu -\mu/2) \le e^{-\mu/8}.
\]
Using the last inequality and the union bound finishes the proof.
\end{proof}

\begin{lemma}\label{no small out-degree}
There is no vertex in $[m+1,n]$ with out-degree strictly smaller than $k-1$ in $\dgnk$.
\end{lemma}

\begin{proof}
For any $x>m$,
\[
\pr(\dout{x} \le k-2)
\le {x-1\choose k-2} \lp \frac{k-2}{x-1}\rp^k
=O(x^{-2}).
\]
Hence
\[
\pr(\exists\,x> m: \dout{x}\le k-2)=O\left(\sum_{x> m}x^{-2}\right)=O(1/m)\to 0. \qedhere
\]
\end{proof}

\begin{proof}[\bf{Proof of Theorem~\ref{thm:mindeg}}]
Let ${\cal S}_n$ denote the set of vertices with out-degree $(k-1)$ and in-degree 0 in $\dgnk$.
Lemmas~\ref{big in-degrees} and \ref{no small out-degree} and the fact $D(n)\le k$ together imply that $\delta_n$ is either $k-1$ or $k$ whp, and
\[
\pr(\delta_n =k-1)= \pr({\cal S}_n\not= \varnothing) +o(1).
\]
Hence it is enough to show 
\[
\pr({\cal S}_n= \emptyset) \sim e^{-(k-1)/2}.
\]
Since  $[m]\cap {\cal S}_n=\emptyset$ whp by Lemma~\ref{big in-degrees}, we may only focus on the set $S_n^*:=S_n\setminus [m]$.

For a given vertex $x$,
\beq\label{in-degree=0}
\pr(\din{x}=0)=\prod_{y> x}\frac{(y-2)^k}{(y-1)^k}= \frac{(x-1)^k}{(n-1)^k}= \left(1+O(x^{-1})\right) (x/n)^k.
\eeq
On the other hand, for every vertex  $x\ge k$,
\beq\label{outdegree=k-1}
\pr(\dout{x}=k-1)={x-1\choose k-1}{k\choose 2} (k-1)!\frac 1{(x-1)^k}= {k \choose 2}\frac{(x-1)_{k-1}}{(x-1)^k} \sim \frac{k(k-1)}{2x}.
\eeq
Since $\din{x}$ and $\dout{x}$ are independent, 
\[
\pr(\din{x}=0,\, \dout{x}=k-1)=\binom{k}{2}\frac{(x-1)_{k-1}}{(n-1)^k}.
\]
More generally, given $\ell\ge 1$, consider $\bold x:=\{x_1>x_2>\cdots>x_{\ell}\}$. Let us compute
\[
\pr(\bold x):=\pr\left(\bigcap_{j=1}^{\ell}\big\{\din{x_j}=0,\, \dout{x_j}=k-1\big\}\right).
\]
On this event, the admissible selections for the vertex $y\in [n]$ form a set 
\[
A(y):=[1,y-1]\setminus \{x_j:\, x_j<y\};
\]
if y selects $x_j$, then the vertex $x_j$ will have a positive in-degree. ``Merging'' \eqref{in-degree=0} and
\eqref{outdegree=k-1} we obtain
\[
\pr(\bold x)=\prod_{y\notin \bold x:\,y>x_{\ell}} \frac{|A(y)|^k}{(y-1)^k}\cdot\prod_{y\in\bold x}\frac{\binom{|A(y)|}{k-1}\binom{k}{2}(k-1)!}{(y-1)^k}.
\]
Telescoping the first product for $y>x_1$ yields
\beq\label{P(boldx)=}
\pr(\bold x)=\prod_{j=1}^{\ell}\frac{(x_1-j)^k}{(n-j)^k}\prod_{y\notin \bold x:\,x_{\ell}<y<x_1}
\!\!\frac{|A(y)|^k}{(y-1)^k}\cdot\prod_{y\in\bold x}\frac{\binom{|A(y)|}{k-1}\binom{k}{2}(k-1)!}{(y-1)^k}.
\eeq
We will need a more transparent asymptotic version of \eqref{P(boldx)=} for the case when $x_{\ell}$
is large. First,
\[
\prod_{j=1}^{\ell}\frac{(x_1-j)^k}{(n-j)^k}=(1+O(x_1^{-1}))\left(\frac{x_1}{n}\right)^{\!k\ell}.
\]
Similarly,
\[
\prod_{y\in\bold x}\frac{\binom{|A(y)|}{k-1}\binom{k}{2}(k-1)!}{(y-1)^k}=(1+O(1/x_{\ell}))\binom{k}{2}^{\!\!\ell}\prod_{j=1}^{\ell}\frac{1}{x_j}.
\]
Further, for $y\in (x_{j+1},x_j)$ we have $|A(y)|=y-(\ell+1-j)$; therefore
\begin{align*}
\prod_{y\in (x_{j+1},x_j)}\frac{|A(y)|^k}{(y-1)^k}&=\prod_{y\in (x_{j+1},x_j)}\left(\frac{y-1-(\ell-j)}{y-1}\right)^k\\
&=\prod_{y\in (x_{j+1},x_j)}\exp\left(-\frac{k(\ell-j)}{y-1}+O(y^{-2})\right)\\
&=\exp\left(-\sum_{y\in (x_{j+1},x_j)}\frac{k(\ell-j)}{y-1}+O(x_{j+1}^{-1})\right)\\
&=\exp\left(-k(\ell-j)\log\frac{x_j}{x_{j+1}}+O(x_{j+1}^{-1})\right)\\
&=(1+O(x_{j+1}^{-1})) \left(\frac{x_{j+1}}{x_j}\right)^{k(\ell-j)}.
\end{align*}
Therefore, for $x_{\ell}\to\infty$ however slowly, we have
\begin{align}\label{P(x)appr}
\pr(\bold x)
&=(1+O(x_{\ell}^{-1}))\binom{k}{2}^{\ell}\left(\frac{x_1}{n}\right)^{k\ell}\left(\prod_{j=1}^{\ell}\frac{1}{x_j}\right) \cdot \prod_{j=1}^{\ell-1}\left(\frac{x_{j+1}}{x_j}\right)^{k(\ell-j)}	\notag	\\
&=(1+O(x_{\ell}^{-1}))\binom{k}{2}^{\ell}n^{-k\ell}\prod_{j=1}^{\ell}x_j^{k-1}.
\end{align}
Using~\eqref{P(x)appr}, we get
\begin{align*}
\mean\left[\binom{|{\cal S}_n^*|}{\ell}\right]&=\sum_{x_1>\cdots>x_{\ell}>m}\,\pr(\bold x)\\
&=(1+O(m^{-1}))\binom{k}{2}^{\ell}\sum_{x_1>\cdots>x_{\ell}>m}\prod_{j=1}^{\ell}\frac{1}{n}\left(\frac{x_j}{n}\right)^{k-1}\\
&=(1+O(m^{-1}))\binom{k}{2}^{\ell}\idotsint\limits_{1\ge z_1>\cdots>z_{\ell}} \prod_{j=1}^{\ell} z_j^{k-1}\,dz_1\cdots dz_{\ell}\\
&=\frac{1+O(m^{-1}))}{\ell!}\binom{k}{2}^{\ell}\left(\int_0^1 z^{k-1}\,dz\right)^{\ell}
=(1+O(m^{-1}))\frac{\left(\frac{k-1}{2}\right)^{\ell}}{\ell!}.
\end{align*}
Therefore $|{\cal S}_n^*|$ (whence $|S_n|$) converges in distribution to a Poisson random variable with mean $(k-1)/2$, which finishes the proof.
\end{proof}

\begin{remark}
The proof of Theorem~\ref{thm:mindeg} shows that the number of vertices with degree smaller than $k$ is bounded whp. What about the number of vertices with degree $k+j$, with $j\ge 0$ fixed?
Since $\pr(\din{x}<k)=O(x^{-1})$ for all $x\in [n]$, almost all the vertices have out-degree $k$. Let $X_j$ denote the number of vertices with in-degree $j$. Standard computations of the first two moments give
\[
X_j \sim \frac{n}{k+1} \lp \frac{k}{k+1}\rp^j
\]
whp. Thus, as opposed to preferential attachment graphs, where the degrees follow power-law distribution, the degrees in $G_{n,k}$ have geometric distribution; the number of vertices with total degree $k+j$, scaled by $n$, is asymptotically $\frac{1}{k+1} \lp \frac{k}{k+1}\rp^j$ whp.
\end{remark}

\subsection{Connectivity of $G_{n,k}$}

\begin{proof}[{\bf Proof of Theorem~\ref{thm:connectivity}}] 
Since $\kappa_n\le\delta_n$ always, it suffices to prove that $\lim_{n\to\infty}\pr(\kappa_n<\delta_n)=O(n^{-1/k})$.

\medskip

\noindent \textbf{(1)} Let us first prove 
\begin{equation}\label{P(kap<del=k-1)}
\pr(\kappa_n<\delta_n=k-1)=O(n^{-2/k}).
\end{equation}
Since $G_{n,k}$ is connected, we have $\kappa_n\ge 1$. Hence it suffices to consider $k\ge 3$.
By the definition of connectivity, $\pr(\kappa_n<k-1)=\pr({\cal U}_n)$,
where ${\cal U}_n$ is the event that there exists a partition  $[n]=X_1\sqcup X_2\sqcup X_3$  such that, denoting $\nu_j=|X_j|$:
\begin{enumerate}[noitemsep, topsep=0pt]
\item[(i)] $1\le \min\{\nu_1,\nu_2\}$ and $1\le\nu_3\le k-2$,
\item[(ii)] no vertex of $X_1$ ($X_2$ resp.) selects a vertex from $X_2$ ($X_1$ resp.).
\end{enumerate}
In fact, it suffices to consider only $\nu_3=k-2$ and from now on we assume this is the case.
 By the union bound and the definition of $G_{n,k}$,
\begin{equation}\label{P(Un)<X1,X2,X3}
\pr({\cal U}_n)\le\sum_{X_1,X_2,X_3}
\prod_{j=1}^2\prod_{x\in X_j}\frac{(\nu_j+\nu_3-r(x; X_j\sqcup X_3))^k}{(x-1)^k}.
\end{equation}
Here, for $j=1,2$,  $r(x;X_j\sqcup X_3)$ is the rank of $x\in X_j$ in the set $X_j\sqcup X_3 $. (E.g. $r(x:X_j\sqcup X_3)=1$ if $x$ is the largest vertex in $X_j\sqcup X_3$. )  
So  the last fraction is the probability
that vertex $x\in X_j$ selects the vertices $y<x$ exclusively from $X_j\sqcup X_3$. Now $r(x:X_j\sqcup X_3)\ge r(x;X_j)$, the rank of $x\in X_j$ in $X_j$, and $r(x;X_j)$ runs from $1$ to $\nu_j$ as $x$ goes through the vertices of $X_j$ in the decreasing order.
So
\begin{align*}
\prod_{j=1}^2\prod_{x\in X_j}(\nu_j+\nu_3-r(x;x_j\sqcup X_j))^k&\le \left[\prod_{j=1,2}(\nu_j+\nu_3-1)_{\nu_j}\right]^k\\
&\le \gamma_1 \bigl[(\nu_1+\nu_3-1)!\, (\nu_2+\nu_3-1)!\bigr]^k,
\end{align*}
where $\gamma_1$ and $\gamma_r$ below stand for constants dependent on $k$ only. Since
\begin{equation}\label{(n-1-nu_3)!}
\begin{aligned}
\prod_{j=1}^2\prod_{x\in X_j}\frac{1}{(x-1)^k}&=\frac{1}{[(n-1)!]^k}\prod_{x\in X_3}(x-1)^k\\
&\le \frac{[(n-1)_{\nu_3}]^k}{[(n-1)!]^k}=\frac{1}{[(n-1-\nu_3)!]^k},
\end{aligned}
\end{equation}
the product in \eqref{P(Un)<X1,X2,X3} is bounded by a function dependent on $\nu_j$ only. Therefore,  \eqref{P(Un)<X1,X2,X3} 
is transformed into
\begin{align}
\pr({\cal U}_n)&\le_b\sum_{\nu_1+\nu_2+\nu_3=n\atop \min\{\nu_1,\nu_2\}\ge 1} 
\frac{n!}{\nu_1!\nu_2!\nu_3!}\left[\frac{(\nu_1+\nu_3-1)!\, (\nu_2+\nu_3-1)!}{(n-1-\nu_3)!}\right]^k\notag\\
&\le_b n^{\nu_3}\!\!\!\!\!\sum_{\nu_1+\nu_2=n-\nu_3\atop \min\{\nu_1,\nu_2\}\ge 1}\!\! \frac{(n-\nu_3)!}{\nu_1!\nu_2!}
\cdot\left[\frac{(\nu_1+k-3)!\,(n-\nu_1-1)!}{(n-\nu_3-1)!}\right]^k.\label{<bSum}
\end{align}
Here $\nu_3=k-2$. The term corresponding to $j=\nu_1-\nu_3$ (with fixed $j$) is in the order of $O(n^{\nu_3+\nu_1-jk})=O(n^{(2-j)(k-1)-2})$, which is $O(n^{-2})$ for $j\ge 2$. So it is enough to consider the case $j\le 1$, i.e., the case $\nu_1\le k-1$.

Consider $\nu_1\le k-1$. We will be applying the union bound again, but we will need to use some additional constraints the  
partitions $(X_1,X_2,X_3)$ must meet for the sets $X_1$ and $X_2$ to be separated by removal of the set $X_3$. Let $n_0=0$ and $n_j=\floor{n^{j/k}}$ for $j=1,\dots,k$. Also let $I_j=[n_{j-1}+1,n_j]$. We can choose $X_1$ by first specifying the numbers $a_i=|X_1\cap I_i|$ (subject to the constraint $a_1+\cdots+a_k=\nu_1$),  and then choosing $a_i$ numbers from $I_i$ for each $i$.  For a particular sequence $\bold a:=(a_1,\dots,a_k)$, there are at most $n^{\sum ia_i/k}$ choices for $X_1$. Now fix $X_1$ and $X_3$. By the definition of the event ${\cal U}_n$, each $x\in X_1$ selects (with repetition) $k$ members exclusively from $X_1\cup X_3$, which happens with probability
\[
\left(\frac{|X_1|+|X_3|}{x-1}\right)^k\le \left(\frac{2k}{n_{j-1}-2}\right)^k=O\bigl(n^{-(j-1)}\bigr),\quad\text{if }x\in X_1\cap I_j.
\]
Therefore, for the event $\mathcal A_1:=\{\text{no member of }X_1\text{ selects a member of }X_2\}$ we have
\begin{equation}\label{noX1X2}
\pr(\mathcal A_1)=O\left(\prod_{j=1}^k \left(n^{-(j-1)}\right)^{a_j}\right)=O\left(n^{-\sum_ja_j(j-1)}\right).
\end{equation}
Consider now the event $\mathcal A_2:=\{\text{none of the members of }X_1\text{ is chosen by a member of }X_2\}$.
Let $y\in X_2\cap I_j$. The probability that $y$ does not choose a member from $X_1$ is at most
\[
\left( 1-\frac{a_1+\cdots+a_{j-1}}{y-1} \right)^k \le \exp(-k(a_1+\cdots+a_{j-1})/y).
\]
Since these events for different $y\in X_2\cap I_j$ are independent, we have
\begin{align*}
&\pr(\text{no } y\in X_2\cap I_j \text{ chooses a member from }X_1) 
\le \prod_{y\in X_2\cap I_j} \exp\left(-\frac{k}{y}\sum_{t=1}^{j-1}a_t\right)\\
&=O\!\left(\!\exp\left(\!-k\left(\sum_{t=1}^{j-1}a_t\right)\sum_{y=n_{j-1}}^{n_j-1}\frac{1}{y}\right)\!\right)
=O\!\left(\!\exp\left(\!-\left(\sum_{t=1}^{j-1}a_t\right)\log n\right)\right)
=O\left(n^{-\sum_{t=1}^{j-1}a_t}\right).
\end{align*}
Consequently,
\begin{align*}
\pr({\cal A}_2) &= \prod_{j=1}^k\pr\bigl(\text{no } y\in X_2\cap I_j \text{ chooses a member from }X_1\bigr)\\
&=O\left(\prod_{j=1}^k n^{-\sum_{j=1}^k (k-j)a_j}\right).
\end{align*}
Finally,
\begin{align}
\pr({\cal A}_1 \cap {\cal A}_2) &=\pr({\cal A}_1)\pr({\cal A}_2)= O\left( n^{\sum_{j=1}^{k-1} a_j(1-j)-\sum_{j=1}^{k-1} (k-j)a_j} \right)\notag\\
&=O\left( n^{-\sum_{j=1}^k a_j(1-k)} \right)=O\left( n^{\nu_1(1-k)} \right).\label{P(A1capA2)}
\end{align}

Since the set $X_3$ can be chosen in $\binom{n}{k-2}=O(n^{k-2})$ ways, the total number of ways to choose
$X_3$ and $X_1$, for a given $\bold a$, is of order $n^{L(\bold a)}$, where $L(\bold a):=k-2+\sum_j ja_j/k$. So the contribution 
of the summands with $\nu_1\le k-1$ is of order $\sum_{\bold a}n^{L^*(\bold a)}$, $L^*(\bold a):=\nu_1(1-k)+L(\bold a)$. 
Note that $\sum_j ja_j/k \le \nu_1$ with equality holding when $a_k=\nu_1$. 
Consequently, if $\nu_1\ge 2$ or $\nu_1=1$ and $a_k=0$ then
in the first case we have $L^*(\bold a)\le -(k-2)$, and in the second case
$L^*(\bold a)\le -1+\frac{k-1}{k}=-\frac{1}{k}$.

Consider the complementary case: $\nu_1=1$ and $a_k=1$.
So $X_1=\{x\}$ and $x\ge n_k:=\lfloor n^{(k-1)/k}\rfloor\to\infty$. 
Since $x>1$,  we have $x>\min X_3$.  Otherwise the non-empty set of all elements smaller
than $x$ would be entirely contained in $X_2$, and some of them would be selected by $x$, making the partition 
$(X_1,X_2,X_3)$ non-admissible.
Given the element $x\ge n_{k}$, an element $y<x$ can be chosen in $x-1$ ways, and the 
additional $k-3$ elements, to be included in $X_3$, in  $\binom{n-2}{k-3}$ ways. 
The number of ways for $x$ to make $k$ (possibly repeated) selections from the chosen $k-2$ elements in $X_3$ is at most the number of ways to allocate $k$ distinguishable balls among $(k-2)$ boxes, which is $\binom{2k-3}{k-3}$, and each such selection 
has probability $(x-1)^{-k}$. Hence the contribution of $\nu_1=a_k=1$ to $\pr({\cal U}_n)$ is at most of order
\[
\binom{n}{k-3}\sum_{x=n_{k}}^n  x^{-k+1}=O(n^{k-3})\sum_{x=n_{k}}^{\infty} x^{1-k}=o(n^{-1}).
\]
So the overall contribution of the summands with $\nu_1\le k-1$ to $\pr(\mathcal U_n)$ is $O(n^{-2/k})$.
It follows that $\pr({\cal U}_n)=O(n^{-2/k})$ as well. This proves \eqref{P(kap<del=k-1)}.

\medskip

\noindent \textbf{(2)} Now let us prove
\begin{equation}\label{P(kap<del=k)}
\pr(\kappa_n<\delta_n=k)=O(n^{-1/k}).
\end{equation}
The proof runs parallel to the part {\bf (1)\/}. We have $\pr(\kappa_n<\delta_n=k)\le\pr({\cal V}_n)$, with ${\cal V}_n$ defined
similarly to ${\cal U}_n$, except that now $\nu_3:=|X_3|=k-1$.  In addition, $|X_1|>1$ and $|X_2|>1$: if $X_i=\{x\}$, then all the $x$'s neighbors are in $X_3$, and this is impossible on the event $\{\delta_n=k\}$. The counterpart of \eqref{<bSum} is 
\begin{equation}\label{<bSum*}
\pr({\cal V}_n)\le_b n^{\nu_3}\!\!\!\!\!\sum_{\nu_1+\nu_2=n-\nu_3\atop \min\{\nu_1,\nu_2\}\ge 1,}\!\! \frac{(n-\nu_3)!}{\nu_1!\nu_2!}
\cdot\left[\frac{(\nu_1+k-2)!\,(n-\nu_1-1)!}{(n-\nu_3-1)!}\right]^k,\quad (\nu_3=k-1).
\end{equation}
The terms corresponding to $j=\nu_1-\nu_3$ (with fixed $j$) are in the order of $O(n^{(2-j)(k-1)})$, i.e. for bounded $j\ge 3$
they are dominated
by the geometric progression with denominator and first term of order $n^{-k+1}$. So it is enough to consider the case $j\le 2$, i.e. $\nu_1\le k+1$, so that $|X_2|\ge n-2k$.

The bound \eqref{P(A1capA2)} continues to hold, but the number of ways to choose $X_3$ is now $O(n^{k-1})$. So  the
totall number of ways to choose $X_3$ and $X_1$ for a given $\bold a$ is now $\mathcal L(\bold a):=k-1+\sum_j ja_j/k$.
Hence the contribution of the summands with $\nu_1\le k+1$ is of order $\sum_{\bold a}n^{\mathcal L^*(\bold a)}$, $\mathcal L^*(a)=\nu_1(1-k)+\mathcal L(\bold a)$. Arguing as in the part {\bf (1)\/}, we obtain that $\mathcal L^*(\bold a)\le 5-2k\le -1$ if 
$\nu_1\ge 3$, and $\mathcal L^*(\bold a)\le -1/k$ if $\nu_1=2$ and $a_k\le 1$.

Consider the complementary case $\nu_1=2$, $a_k=2$; so $X_1=\{x_1, x_2\}$, and $\min(x_1,x_2)\ge n_k=\lfloor n^{(k-1)/k}\rfloor$. Suppose $x_1>x_2$. Since $\nu_3=k-1$, on the event  ``$\delta_n=k$'' {\bf (a)\/} the vertices $x_1$ and $x_2$ form an edge, i.e.
$x_1$ selects $x_2$, and {\bf (b)\/} $x_1$ ($x_2$ resp.)  makes $(k-1)$ ($k$ resp.) selections from the $(k-1)$ elements of $X_3$. The probability of this event is of order
\[
n^{k-2}\sum_{x_1\ge x_2\ge n_k} x_1^{-k+1} (x_2)^{-k}\le_b n^{k-2} \left. y^{-2k+3}\right|_{y=n_k}=o(n^{-1}). 
\]
Therefore
$
\pr({\cal V}_n)=O(n^{-1/k}),
$
which proves \eqref{P(kap<del=k)}.
\end{proof}

\section{$G_{n,k}$ expands.}\label{sec:expander}

In this section we prove Theorem~\ref{expansion} and Theorem \ref{walk}. Theorem \ref{expansion} states that, for $k\ge 2$, we have
\[
\lim_{n\to\infty}\!\pr\!\left\{\exists\, X\subseteq [n] : 1\le |X|\le n/2\text{ and }|N(X)|\le\rho |X| \right\}=0
\]
for any $\rho<\rho(k)$, where $\rho(k)$ is the positive root of
\begin{align*}
&\log 2-\frac{1}{2}\bigl[\rho\log \rho+(1-\rho)\log(1-\rho)\bigr]\\
&\qquad\qquad\qquad\qquad+k \left[\frac{1+\rho}{2}\log(1+\rho)-\rho\log \rho-\left(1-\frac{\rho}{2}\right)\log(2-\rho)\right]=0.
\end{align*}
In words, whp $G_{n,k}$ is vertex-expanding with rate $\rho$ at least.

\begin{proof}[{\bf Proof of Theorem~\ref{expansion}}]
Let $P_n$ denote the probability in question. Let $Y:=N(X)$. By the definition of $Y$, no $x\in X$ makes any selection outside 
$X\sqcup Y$, and no $z\in (X\sqcup Y)^c$ makes any selection in $X$. So,  by the union bound, we have
\begin{equation}\label{P(shrink)}
\begin{aligned}
P_n&\le \sum_{X,Y:|Y|\le \rho |X|\atop
|X|+|Y|\ge k}\prod_{x\in X}\frac{\bigl(|X|+|Y|-r(x;X\sqcup Y)\bigr)^k}{x^k}\,\,\cdot\prod_{z\in (X\sqcup Y)^c}
\frac{\bigl(|X^c|-r(z;X^c)\bigr)^k}{z^k}\\
&\le\sum_{X,Y:|Y|\le \rho |X|\atop
|X|+|Y|\ge k}\prod_{x\in X}\frac{\bigl(|X|+|Y|-r(x;X)\bigr)^k}{x^k}\,\,
\cdot\prod_{z\in (X\sqcup Y)^c}
\frac{\bigl(|X^c|-r(z;(X\sqcup Y)^c)\bigr)^k}{z^k}.
\end{aligned}
\end{equation}
For $x$ running through $X$ in the decreasing order, $r(x;X)$ runs from $1$ to $|X|-1$. Likewise in the second product
$r(z;(X\sqcup Y)^c))$ runs from $1$ to $|(X\sqcup Y)^c|-1$. Therefore
\begin{align*}
\prod_{x\in X}\bigl(|X|+|Y|-r(x;X)\bigr)&=\bigl(|X|+|Y|-1\bigr)_{|X|-1}=\frac{(|X|+|Y|-1)!}{|Y|!},\\
\prod_{z\in (X\sqcup Y)^c}
\bigl(|X^c|-r(z;(X\sqcup Y)^c)\bigr)&=(n-|X|-1)_{n-|X|-|Y|-1}=\frac{(n-|X|-1)!}{|Y|!}.
\end{align*}
In addition
\begin{align*}
\prod_{x\in X} x \cdot \prod_{z\in (X\sqcup Y)^c}\!\!\!\!\!z &= \prod_{x\in X}\! x \cdot \left(\frac{n!}{\prod\limits_{z\in X\sqcup Y} z}\right)
=\frac{n!}{\prod\limits_{z\in Y} z}\ge \frac{n!}{(n)_{|Y|}}=(n-|Y|)!.
\end{align*}
So denoting $|X|=\mu$, $|Y|=\nu$, we have
\begin{equation}\label{Pn<expl}
P_n\le \sum_{0<\mu\le n/2\atop \mu+\nu\ge k,\,\nu\le\rho \mu}\frac{n!}{\mu!\,\nu!\,(n-\mu-\nu)!}
\times\left(\frac{(\mu+\nu-1)!\,(n-\mu-1)!}{(\nu !)^2\,(n-\nu)!}\right)^k.
\end{equation}
Let $R_n(\mu,\nu)$ stand for the $(\mu,\nu)$-summand. 

{\bf (i)\/} Consider $\mu\le \mu_n:=\lfloor \eps\log n\rfloor$. Since $\mu_n^2\ll n$, and
$\nu\le \rho\mu\le \mu_n$, we have
\[
\frac{n!}{(n-\mu-\nu)!}\le_b n^{\mu+\nu},\quad \frac{(n-\mu-1)!}{(n-\nu)!}\le_b n^{\nu-\mu-1},
\]
so that 
\[
\frac{n!}{(n-\mu-\nu)!}\cdot\left[\frac{(n-\mu-1)!}{(n-\nu)!}\right]^k\le_b n^{-\mu\Delta-k},\quad\Delta:=(k-1)-(k+1)\rho.
\]
Notice that $\Delta>0$ if $\rho<\frac{k-1}{k+1}$ and $\min_{k\ge 2}\frac{k-1}{k+1}=1/3$. Also
\begin{align*}
\frac{1}{\mu!\,\nu!}\left(\frac{(\mu+\nu)!}{(\nu!)^2}\right)^k&\le\left(\frac{\mu!}{\nu!}\right)^k\binom{\mu+\nu}{\mu}^k
\le \mu^{k\nu}\cdot 2^{k(\mu+\nu)}\le \mu^{k\rho\mu}\cdot 2^{k(\mu+\nu)}.
\end{align*}
So, uniformly for $\mu\le\mu_n$, we obtain
\begin{align*}
R_n(\mu,\nu)&\le_b\left(\frac{\mu^{k\rho}}{n^{\Delta}}\right)^{\mu}\cdot\left(\frac{2^{\mu+\nu}}{n}\right)^k
\le n^{-\mu(\Delta +o(1))}n^{-k+O(\eps)}
\end{align*}
Consequently,
\begin{equation}\label{musmall}
\sum_{1\le\mu\le\mu_n\atop \nu\le\rho\mu}R_n(\mu,\nu)=O\bigl(n^{-k+O(\eps)}\bigr).
\end{equation}

{\bf (ii)\/} Turn to $\mu\in [\mu_n,n/2]$. Call a pair $(\mu,\nu)$ admissible if it is in the summation range. Suppose that $(\mu-1,\nu+1)$ is admissible as well.  Then
\[
\frac{R_n(\mu-1,\nu+1)}{R_n(\mu,\nu)}=\frac{\mu}{\nu+1}\left[\frac{(n-\mu)(n-\nu)}{(\nu+1)^2}\right]^k\ge \left(\frac{n-\mu}{\mu}\right)^{2k}\ge 1,
\]
as $\mu\le n/2$. So, given $s\in [(1+\rho)\mu_n, (1+\rho)n/2]$,
\begin{align*}
&\max \{R_n(\mu,\nu): \mu+\nu =s,\,\mu\in [\mu_n,n/2],\, \nu\le\rho\mu\}:=R_n(\mu(s),\nu(s))=:R_n(s),\\
&\qquad\qquad\qquad\quad\mu(s):=\left\lceil\frac{s}{1+\rho}\right\rceil,\,\nu(s):=s-\mu(s).
\end{align*}
Introducing $r=r(s)=\nu(s)/\mu(s)$, we have $\mu(s)=s/(1+r)$, $\nu(s)=rs/(1+r)$, and  
\begin{equation}\label{rapproxrho}
r=\rho+O(\mu_n^{-1})=\rho+O(\log^{-1}n).
\end{equation}
Scaling $s$, set $x=s/n$, so that $\mu(s)/n=x/(1+r)$, $\nu(s)/n=xr/(1+r)$; then the constraint $\mu(s)\le n/2$ translates into $x\le x_r:=(1+r)/2$. 
Let us upper-bound $R_n(s)=R_n(\mu(s),\nu(s))$ in terms of the parameter $x$. For the trinomial factor in $R(\mu,\nu)$, the 
$(\mu,\nu)$-summand in \eqref{Pn<expl}, we use the multinomial inequality and then plug in $\mu(s)$, and $\nu(s)$ to obtain
\begin{equation}\label{H}
\begin{aligned}
&\left.\frac{n!}{\mu!\,\nu!\,(n-\mu-\nu)!}\right|_{\mu=\mu(s)\atop \nu=\nu(s)}\le\left.\frac{n^n}{\mu^{\mu}\,\nu^{\nu}\,(n-\mu-\nu)^{n-\mu-\nu}}\right|_{\mu=\mu(s)\atop \nu=\nu(s)}=\left.:\exp\bigl(n H(x)\bigr)\right|_{x=s/n},\\
&\qquad\qquad H(x)=-\frac{x}{1+r}\log\frac{x}{1+r} -\frac{rx}{1+r}\log\frac{rx}{1+r}-(1-x)\log (1-x).
\end{aligned}
\end{equation}
For the remaining factor we use the Stirling formula for the factorials involved and get
\begin{equation}\label{calH}
\begin{aligned}
&\left.\frac{(\mu+\nu-1)!\,(n-\mu-1)!}{(\nu !)^2\,(n-\nu)!}\right|_{\mu=\mu(s)\atop \nu=\nu(s)}\le_b\left.\frac{(\mu+\nu)^{\mu+\nu}
(n-\mu)^{n-\mu}}{\nu^{2\nu}(n-\nu)^{n-\nu}}\right|_{\mu=\mu(s)\atop \nu=\nu(s)}=\left.:\exp\bigl(n\mathcal H(x)\bigr)\right|_{x=s/n}\\
&\qquad\qquad\quad\mathcal H(x)=x\log x+\left(1-\frac{x}{1+r}\right)\log\left(1-\frac{x}{1+r}\right)\\
&\qquad\qquad-2\frac{rx}{1+r}\log\left(\frac{rx}{1+r}\right)-
\left(1-\frac{rx}{1+r}\right)\log\left(1-\frac{rx}{1+r}\right).
\end{aligned}
\end{equation}
Therefore, uniformly for $s\in [(1+\rho)\mu_n,(1+\rho)n/2]$,
\[
R_n(s)\le_b \exp\bigl[n(H(x)+\mathcal H(x))\bigr],\quad (x=s/n).
\]
We remind that $s/n\le x_r:=(1+r)/2$ and $r$ satisfies \eqref{rapproxrho}, thus is asymptotic to $(1+\rho)/2$. Our task is to prove that
$K(x):=H(x)+k\mathcal H(x)<0$ for $x\in (0, x_r]$, provided that $\rho$ is below $r(k)$. 

Let us show that $K''(x)>0$ for $x\in (0,x_r]$, and all $k\ge 2$ if $r<2/3$.  Using \eqref{H} for $H(x)$, we obtain without much work
that $H''(x)=-\frac{1}{x(1-x)}$. Derivation of $\mathcal H''(x)$ is more involved, but the final formula is comfortingly simple as well:
\[
\mathcal H''(x)=\frac{(1+r)^2}{x(1+r-x)(1+r-rx)};
\]
therefore
\[
K''(x)=-\frac{1}{x(1-x)}+k\frac{(1+r)^2}{x(1+r-x)(1+r-rx)}.
\]
Consequently, $K''(x)>0$ for $x\le x_r$ iff
\[
\frac{r}{(1+r)^2}\frac{x^2}{1-x}<k-1,\quad \forall\, x\in (0,x_r]=(0,(1+r)/2]
\]
whence iff $\frac{r}{2(1-r)}<1$, i.e.\ $r<2/3$. This condition holds for large enough $n$ if our {\it fixed\/} $\rho$ is strictly below $2/3$.
Strict convexity of $K(x)$ on $(0,x_r]$ implies that $K(x)$ is negative on this interval provided that $K(0)\le 0$ and $K(x_r)<0$.
The former condition is met since $K(0)=0$, and the latter condition is met for large enough $n$ if $K((1+\rho)/2)<0$, {\it and\/}
$\rho<2/3$, which holds since we consider $\rho<1/3$ only.

Therefore, whp $G_{n,k}$ is a $\rho$-vertex expander if $\rho<\rho^*(k)$, where $\rho^*(k)=\min(1/3,\rho(k))$, and $\rho(k)$ is
the root of $H\bigl(\frac{1+\rho}{2}\bigr)+k\mathcal H\bigl(\frac{1+\rho}{2}\bigr)=0$, or explicitly
\begin{align*}
&\log 2-\frac{1}{2}\bigl[\rho\log \rho+(1-\rho)\log(1-\rho)\bigr]\\
&\qquad\qquad\qquad +k\left[\frac{1+\rho}{2}\log(1+\rho)-\rho\log \rho-\left(1-\frac{\rho}{2}\right)\log(2-\rho)\right]=0.
\end{align*}
An elementary calculus reveals  that $\rho(k)$ increase to $\rho( \infty)=r^*\approx 0.252<1/3$, so that $\rho^*(k)=\rho(k)$.

\end{proof}

Let $\diam(G_{n,k})$ denote the diameter of $G_{n,k}$. A standard argument gives the following corollary.

\begin{corollary} \label{cor:diameter}
For every $\rho<\rho(k)$, whp $\diam(G_{n,k}) = O(\log_{1+\rho}n)$. \qed
\end{corollary}

\begin{remark}
$G_{n,1}$ is a random recursive tree and the height of this tree is found by Pittel~\cite{Pittel} as $e\log n$ asymptotically. Hence, $\diam(G_{n,1}) =\Theta(\log n)$ whp. For $k\ge 2$, in addition to the upper bound given in the previous result, we have the lower bound $(1+o(1))\log n/\log\log n$ for $\diam(G_{n,k})$. This follows from the following easy observation coupled with the fact that the maximum degree of $G_{n,k}$ is asymptotically $k\log n$ whp.

Let $G$ be a connected graph with maximum degree $\Delta \ge 2$. Then the diameter of $G$ is at least $(\log n/ \log \Delta)-1$. Indeed, denoting by $B(x,i)$ the set of vertices that are of distance at most $i$ from $x$, we have
\[
|B(x,i)| \ \le \ 1+\Delta+\cdots+\Delta^i \  \le \  \Delta^{i+1}.
\]
Note that if the diameter of $G$ is $d$, then $|B(x,d)|=n$ and hence $n\le \Delta^{d+1}$, which is equivalent to the claim above.
\end{remark}

We use Theorem \ref{expansion} to show that whp $G_{n,k}$ is edge-expanding as well, with rate of order $\log^{-1} n$,
at least. 
Recall that the conductance of $G$, denoted $\Phi(G)$, is defined as $\min_{S\subset V}\Phi(S)$, where
\[
\Phi(S):=\frac{|\nabla(S)|\cdot |E|}{d(S)\cdot d(V\!\setminus\! S)},\quad d(T):=\sum_{v\in T}\text{deg}_G(v),
\]
and $\nabla(S)$ is the set of edges from $S$ to $V\setminus S$. By Theorem \ref{expansion}, whp $|\nabla S|\ge |N(S)|\ge \rho |S|$ for every vertex subset of $G_{n,k}$. Further, the total number of selections made by all $n$ vertices is $kn$ and
the expected total number of repeated selections of a vertex made by another vertex is of order $\sum_{x\in [n]} x^{-1}=
O(\log n)$. Hence $|E(G_{n,k})|\ge kn - O(\log n)$ whp, in which case,
\[
\Phi(S)\, \ge \, b \cdot \frac{n|S|}{D(S) D([n]\!\setminus\! S)},\qquad D(S):=\sum_{x\in S}D(x)
\] 
for an absolute constant $b$.
\begin{lemma}\label{d(X)} There exists $b>0$ such that we have 
\[
\pr\Biggl(\exists\, S\subset [n]:\,D(S)\ge k |S|+b\,|S|\log\frac{n}{|S|}\Biggr)=O(n^{-1}) .
\]
\end{lemma}

A lower bound for the conductance of $G_{n,k}$ follows immediately.

\begin{corollary}\label{Phi(Gnk)>} 
Whp we have $\Phi(G_{n,k})\ge b^* \log ^{-1} n$ for some constant $b^*=b^*(k)>0$.		\qed
\end{corollary}

\begin{remark}
Note that $D(1) \sim \log n$ whp. So whp we have $\Phi(\{1\})= O(\log^{-1}n)$ and hence Corollary~\ref{Phi(Gnk)>} implies $\Phi(G_{n,k})=\Theta(\log^{-1}n)$ whp.
\end{remark}

As discussed in Main results section, Theorem \ref{walk} on convergence rate of the walk distribution to the stationary distribution
follows immediately from Corollary \ref{Phi(Gnk)>}.

\begin{proof} [Proof of Lemma \ref{d(X)}] First of all, $d(x)=d_{in}(x)+d_o(x)$, where $d_{in}(x)$ and $d_o(x)$ are in/out degrees of $x$. Since
$d_o(x)\le k$, we need to prove that, for $b_1$ sufficiently large, 
\begin{equation}\label{di(X)}
\pr\Biggl(\exists S\subset [n]:\,\sum_{x\in S}d_{in}(x)\ge b_1|S|\log\frac{n}{|S|}\Biggr)=O(n^{-1})\,\quad .
\end{equation}
By the definition of $G_{n,k}$, each vertex $y>1$ makes $k$ uniformly random selections, repetitions allowed, among vertices $x\in [y-1]$. So the sequence $\{d_{in}(x)\}_{x\in [n]}$ can be interpreted as occupancy numbers in a non-homogenious
allocation model with $k(n-1)$ distinguishable balls in total, out of which $k$ first balls are thrown (uniformly at random) into bins $\{n-1,\dots, 1\}$, the second batch of $k$ balls is thrown into bins $\{n-2,\dots,1\}$, and the last, $(n-1)$-th, batch
of $k$ balls all go into the box $1$. This is a special case of the balls and bins model, for which the bin occupancy numbers
are known to be negatively associated, Dubhashi and Ranjan \cite{DubRan}. 
Let $D_{j}$ denote the $j$-th largest in-degree in $G_{n,k}$. Obviously, $\sum_{x\in S}\din{x}\le \sum_{j\le |S|} D_{j}$; so it suffices to find a likely upper bound
for $\sum_{j\le J} D_{j}$.

Let $d$ be given. By the definition of $D_j$, the union bound and the negative association of the in-degrees, we have
\begin{align}\label{1/j!}
\pr(D_j\ge d)
&\le \sum_{X\subset [n]: |X|=j}\pr\Bigl(\bigcap_{x\in X}\{\din{x}\ge d\}\Bigr)	\notag\\
&\le \sum_{X\subset [n]: |X|=j}\,\prod_{x\in X}\!\!\pr(\din{x}\ge d)			\notag\\
&\le \frac{1}{j!}\left(\sum_{x\in [n]}\!\!\pr\bigl(\din{x}\ge d\bigr)\right)^j.
\end{align}
For the RHS in \eqref{1/j!} to tend to zero, $d$ has to depend on $j$. To choose $d=d(j)$, first let us bound 
$\pr(\din{x}\ge d)$ by an explicit function of $d$ and $x$.
By the definition of $G_{n,k}$, we have $\din{x}=\sum_{y>x} \xi_y(x)$,
where $\xi_{y}(x)$ is the indicator of the event ``$y$ selects $x$''. We already saw  that, given $x$, the $\xi_y(x)$ are independent,
and
\begin{equation}\label{<k/y}
\pr(\xi_y(x)=1)=1- \left(1- \frac{1}{y-1}\right)^k \le \frac{\max(k,3)}{y}, \quad (y\ge 2). 
\end{equation}
The case $k=2$ is easy to verify directly. Let us consider the case $k\ge 3$. The inequality is equivalent to 
\begin{equation*}
 f(y):=\frac{k}{y}+\left(\frac{y-2}{y-1}\right)^k \ge 1,\quad\forall\,y\ge 2.
\end{equation*}
 We have 
\[
f'(y)=\frac{k}{y^2}\left[\frac{y^2(y-2)^{k-1}}{(y-1)^{k+1}}-1\right],\quad \frac{y^2(y-2)^{k-1}}{(y-1)^{k+1}}-1\to 0,\quad y\to\infty,
\]
and 
\[
\frac{d}{dy}\log\left[\frac{y^2(y-2)^{k-1}}{(y-1)^{k+1}}\right]=\frac{y(k-3)+4}{(y)_3}>0,\quad y>2, 
\]
for $k\ge 3$. So $f'(y)<0$, whence $f(y)>f(\infty)=1$. 
 
Using  $\din{x}=\sum_{y>x} \xi_y (x)$, by Chernoff's method, we have: setting $K:=\max(3,k)$, for every $d>0$, and $z>1$,
\begin{equation*}
\begin{aligned}
\pr(\din{x} \ge d)&\le\frac{\ex\bigl[z^{\din{x}}\bigr]}{z^d}=\frac{\prod_{y>x}\bigl(z\pr(\xi_y(x)=1)+1-\pr(\xi_y(x)=1)\bigr)}{z^d}\notag\\
&\le\frac{\exp\left((z-1)\sum_{y>x}\pr(\xi_y(x)=1)\right)}{z^d}=\frac{\exp\left(K(z-1)\sum_{y>x}y^{-1}\right)}{z^d}\notag\\
&=z^{-d}\exp\Bigl(K(z-1)S(x)\Bigr), \quad S(x):=\sum_{y>x}y^{-1}.
\end{aligned}
\end{equation*}
Since $S(x)\le \log\frac{n+1}{x+1}$, we get from the estimate above that 
\[
\pr(\din{x}\ge d)\le z^{-d}\left(\frac{x+1}{n+1}\right)^{-K(z-1)},\quad \forall\,z>1.
\]
Consequently
\begin{align*}
&\qquad\quad\sum_{x\in [n]}\!\!\pr\bigl(\din{x}\ge d\bigr)\le z^{-d}\sum_{x=1}^{n+1}\left(\frac{x+1}{n+1}\right)^{-K(z-1)}\\
&\le z^{-d} (n+1)\int_0^1 \eta^{-K(z-1)}\,d\eta\le\frac{n+1}{z^d \bigl[1-K(z-1)\bigr]},
\end{align*}
provided that $K(z-1)<1$. The RHS attains its minimum at 
$
z^*=\frac{d(K+1)}{(1+d)K},
$
and $z^*>1$ if $d>K$. For such $d$, we have
\begin{align*}
\sum_{x\in [n]}\!\!\pr\bigl(\din{x}\ge d\bigr)&\le \left.\frac{n+1}{z^d \bigl[1-K(z-1)\bigr]}\right|_{z=z^*}
\le nd \left(\frac{K}{K+1}\right)^d.
\end{align*}
So, recalling \eqref{1/j!} and using $j!\ge (j/e)^j$, we obtain
\[
\pr(D_j\ge d)\le \frac{1}{j!}\left[nd \left(\frac{K}{K+1}\right)^d\right]^j\le \left[\frac{en}{j}\cdot d\left(\frac{K}{K+1}\right)^d\right]^j.
\]
A standard argument shows that for 
\[
d=d_j:=3\frac{\log(an/j)}{\log(1+1/K)},
\]
and $a>e$ such that $3a^{-2}\log a<1$, we have
\[
\pr(D_j\ge d_j)\le \left(\frac{3\log(an/j)}{(an/j)^2}\right)^j.
\]
We conclude that
\[
\pr\left(\bigcup_{j=1}^n\{D_j\ge d_j\}\right)\le \sum_{j=1}^n \pr(D_j\ge d_j)=O\bigl(n^{-2}\log n\bigr).
\]
Consequently, with probability $1-O\bigl(n^{-2}\log n\bigr)$,  for every $S\subset [n]$ we have
\begin{align*}
\sum_{x\in S}d_{in}(x)&\le \frac{3}{\log\frac{K+1}{K}}\sum_{j=1}^{|S|}\log\left(\frac{an}{j}\right)\le  \frac{3}{\log\frac{K+1}{K}}
\left[|S|\log\left(\frac{an}{|S|+1}\right)+|S|+1\right]\\
&\le b\left[|S|\log\left(\frac{n}{|S|}\right)+|S|\right],
\end{align*}
for $b>0$ sufficiently large.
\end{proof}

\section{Bootstrap percolation}\label{sec:BP}
In this section we prove Theorems \ref{thm:percolation1} and \ref{thm:percolation2} that provide qualitatively matching an upper bound and a lower bound for the percolation threshold.

Since $G_{n,k}$ is connected whp, for $r=1$, as long as there is an initially infected vertex, all the vertices will be eventually infected whp. In fact, by Corollary~\ref{cor:diameter}, all the vertices wil be infected in $O(\log n)$ time. In the rest of the paper we consider only $k\ge 3$ and $2\le r \le k-1$. Note that the minimum degree of $G_{n,k}$ is $k-1$ with positive probability.

For $0<p<1$, let $A_0=A_0(p)$ denote the initial set of infected vertices, which is obtained by including each $x\in [n]$ in $A_0$ with probability $p$, indepedently of each other. In  other words, for any $S\subset [n]$,
\[
\pr(A_0=S)=p^{|S|}(1-p)^{n-|S|}.
\]
For $i\ge 1$, let
\[
A_i=A_{i-1}\cup \{x\,:\, x \text{ has at least $r$ neighbors in } A_{i-1}\},
\]
$B_0=A_0$, and $B_i=A_i\setminus A_{i-1}$ for $i\ge 1$.
Hence $A_i$ is the set of infected vertices by time $i$ and $B_i$ is the set of vertices infected at time $i$.

\subsection{Upper bound for the threshold}

Here we prove Theorem~\ref{thm:percolation1}. We start with the following easy lemma, which plays an important role in the proof of the theorem.

\begin{lemma}\label{first m spreads}
Suppose the first $m$ vertices are infected, where $m\to \infty$ however slowly. Then, whp, all the vertices will be infected eventually. 
\end{lemma}

\begin{proof}
By Lemma~\ref{no small out-degree}, whp, there is no vertex in $[m+1,n]$ whose out-degree is strictly smaller than $k-1$. This means, for any $t>m$, if the first $t-1$ vertices are infected, then $t$ will be infected in the next round. The proof follows from induction on $t$.
\end{proof}

A \textit{perfect $\ell$-ary tree} is an $\ell$-ary tree, where each nonleaf has exactly $\ell$ children and all the leaves are of the same depth.
In other words, the size of generation $j$ is $\ell^j$ for all $0\le j\le H(T)$, where $H(T)$ denotes the height of $T$.

\begin{proof}[\textbf{Proof of Theorem~\ref{thm:percolation1}}] 
 Recall that $p\ge \omega (\log n)^{-\frac{r}{r-1}}$, where $\omega=
(3\log ^{(3)} n\cdot \log^{(4)} n)^{\frac{r}{r-1}}$. Let $m=m(n)\to \infty$ slowly enough to satisfy the condition $\log m\ll (\log n)/\omega$.
Define 
\[
\nu=\lfloor \omega^{\frac{r-1}{r}} (\log \omega)^{-1}\rfloor, \qquad \ell=\ell(n)=\lceil (\log n)/\nu\rceil.
\] 
Pick a vertex $x\in [m]$. 
We will first prove that, whp, there exists a perfect $\ell$-ary tree of height $\nu$ in $G_{n,k}$ rooted at $x$.
We will use this fact to prove that if $p$ satisfies the above condition, then whp the root $x$ gets eventually infected. 
And this will imply that, for $m\to\infty$ sufficiently slow, whp infection spreads to every vertex in $[m]$, whence to every vertex in $[n]$.

Let $x\in [m]$. Assuming that $\log m\ll (\log n)/\nu$, let us partition the set $(m,n]$ into $\nu:=\nu(n)$
 intervals $\mathcal I_j$, such that $\mathcal I_1=(m,n^{1/\nu}]$, $\mathcal I_j=(n^{(j-1)/\nu},n^{j/\nu}]$, $j\in [2,\nu]$.
We grow the tree, rooted at $x$, such that, recursively,
vertices from $\mathcal I_j$, $j\ge 2$, select the vertices from $\mathcal I_{j-1}$, already identified as the preceding generation in the tree. 

Start with the children of the root $x\in [m]$.  A vertex $y\in \mathcal I_1$ selects $x$ with probability
$1-(1-1/(y-1))^k\sim k/y$. The expected number of vertices
from $\mathcal I_1$ that select $x$ is asymptotic to 
\[
\sum_{y=m+1}^{n^{1/\nu}}\frac{k}{y}\sim k\log \lp \frac{n^{1/\nu}}{m} \rp  
\sim k\frac{\log n}{\nu} \sim k\ell.
\]
The number itself is a sum of independent Bernoulli variables. 
Hence, by~\eqref{C3}, this number exceeds $\ell$ with probability at least $1-e^{-c\ell}$ for some constant $c\in (0, (k-1)^2/(3k))$.
On this event, we keep 
$\ell$ vertices as the children of $x$, and discard the rest of the vertices from 
$\mathcal I_1$ that selected $x$. Turn to the second generation. The expected number of vertices from $(n^{1/\nu}, n^{2/\nu}]$ that select the smallest vertex of the first generation is
\[
\sum_{y=n^{1/\nu}+1}^{n^{2/\nu}}\frac{k}{y}\sim k\log\frac{n^{2/\nu}}{n^{1/\nu}}\sim k\ell.
\]
By~\eqref{C3}, with probability at least $1-e^{-c\ell}$, we have a set of vertices of cardinality  between
$\ell$ and $2k\ell$, that selected the smallest vertex of the first generation. 
Call this event $\mathcal E_{2,1}$. On $\mathcal E_{2,1}$, we keep $\ell$ vertices as the children of the smallest vertex and discard the excess vertices. 
There remain $R_1:=n^{2/\nu}-n^{1/\nu}-O(\ell)$ vertices in $\mathcal I_2=(n^{1/\nu}, n^{2/\nu}]$ which may select the second smallest vertex,
given that they did not select the smallest vertex.  Conditioned on $\mathcal E_{2,1}$, their choices continue
to be independent and uniform over the remaining vertices. In particular, a vertex $y$ among the remaining $R$
vertices selects the second smallest vertex in the first generation with (conditional) probability $1-(1-1/(y-2))^k\sim k/y$.
This means that,  conditioned on the event $\mathcal E_{2,1}$, the  set of vertices 
selecting the second smallest vertex has cardinality that {\it stochastically dominates (is dominated by)\/}  the sum of
$R$ Bernoulli random variables with individual probabilities again close to $k/j$, $j\in [n^{2/\nu}-R_1+1,n^{2/\nu}]=[n^{1/\nu}+O(\ell),n^{2/\nu}]$  $\,(j\in (n^{1/\nu}, n^{2/\nu}-O(\ell)], \text{resp.})$. 
So, conditioned on the event $\mathcal E_{2,1}$, the event $\mathcal E_{2,2}=$ ``there are
between $\ell$ and $2k\ell$ vertices  that chose  the second smallest vertex'' has
(conditional) probability at least $1-e^{-c\ell}$. We keep $\ell$ of them attached to the second smallest vertex, discard the remaining selectors, and turn to the third smallest vertex. In $\ell$ steps, we arrive at a
sequence of $\ell$ events $\mathcal E_{2,t}$, such that 
\[
\pr\bigl(\mathcal E_{2,t}\boldsymbol | \mathcal E_{2,1},\dots,\mathcal E_{2,t-1}\bigr)\ge 1-e^{-c\ell},
\]
implying
\[
\pr\Biggl(\,\bigcap_{t=1}^{\ell}\mathcal E_{2,t}\Biggr)\ge \left(1-e^{-c\ell} \right)^{\ell}
\ge 1-e^{-c\ell/2}.
\]
(Note that, conditioned on ${\cal E}_{2,1}\cap\cdots\cap {\cal E}_{2,t-1}$, the number of discarded vertices in $I_2$ lies between $(t-1)\ell$ and $2k(t-1)\ell$. Hence, after step $t-1$, the number of ``available'' vertices in $I_2$ is $n^{2/\nu}-n^{1/\nu}-O(t\ell)$. 
In other words, with high probability, for every child of the root we determine the child's own $\ell$ children, completely determining the second generation. Then we move to the
third generation, with $\ell^2$ attendant $\mathcal E_{3,t}$ events, whose intersection is the event that every member
of the second generation has their own $\ell$ children in the third generation, such that
\[
\pr\Bigg(\,\bigcap_{t=1}^{\ell^2}\mathcal E_{3,t}\Bigg)
\ge (1-e^{-c\ell})^{\ell^2}
\ge 1-e^{-c\ell/2}.
\]
In the same fashion we identify the remaining $4$-th,..., $\nu$-th generations. The corresponding bounds
\[
\pr\Biggl(\,\bigcap_{t=1}^{\ell^{j-1}}\mathcal E_{j,t}\Biggr)
\ge (1-e^{-c\ell})^{\ell^{j-1}}
\ge 1-e^{-c\ell/2},				\quad 4\le j\le \nu,
\]
remain valid if $\nu\ll \sqrt{(\log n)/\log\log n}$ which certainly holds for our choice of $\nu$.
With probability exceeding
\[
\left(1-e^{-c\ell/2}\right)^{\nu} \ge 1-e^{-c\ell/3},
\]
in $\nu$ steps we will build a full $\ell$-ary tree of height $\nu$.

Given the tree, let us see whether the root $x$ is infected via the tree by only the initially infected vertices in the last generation. 
The vertices of the $j$-th generation are infected {\it independently\/} of each other, and with the {\it same\/} probability $p_j$, and $p_0$ is the probability that the root $x$ is eventually infected. 
Since the initially infected set and the $\ell$-ary tree are independent, we have the recurrence equation
\begin{equation}\label{recurr}
p_{j-1}=\pr\Bigl(\text{Bin}(\ell,p_j)\ge r\Bigr),\quad p_{\nu}=p.
\end{equation}

Importantly, the sequence $\{p_j\}_{0\le j\le \nu}$ does not depend either on the root $x\in [m]$ or $m$ itself.
Moreover, by~\eqref{recurr}, $p_0\to 1$ if $\liminf_{n\to\infty}\max\{p_j\ell(n):\,j\in [2,\nu]\}>0$. 
Indeed, in this case, for some sequence $\{n_s\}$ and $j=j(n_s)\in [2,\nu]$, we have $\lim\ell(n_s) p_{j(n_s)}>0$. 
This implies $\liminf p_{j(n_s)-1}>0$, which implies, in its turn, that $\lim p_{j(n_s)-2}=1$, and so $\lim p_0=1$.

Suppose that $\liminf_{n\to\infty}\max\{p_j\ell(n):\,j\in [2,\nu]\}=0$. Using \eqref{recurr}, we have: there is a subsequence
$\{n_s\}$ such that for $n\in\{n_s\}$ sufficiently large,
\[
p_{j-1}\ge \binom{\ell(n)}{r}p_j^r (1-p_j)^{\ell(n)-r}
\ge\frac{1}{2\,r!}\ell(n)^rp_j^r,\quad\forall\, j\in [2,\nu].
\]
Iterating this bound, we get
\begin{align*}
p_1\ge\frac{1}{2r!} (\ell(n)p_2)^r&\ge\frac{1}{(2\,r!)^{1+r}}\ell(n)^{r+r^2} p_3^{r^2}\\
&=\cdots=\frac{1}{(2\,r!)^{1+\cdots+r^{\nu-2}}}\ell(n)^{r+\dots+r^{\nu-1}}p_{\nu}^{r^{\nu-1}}\\
&=(2\,r!)^{-\frac{r^{\nu-1}-1}{r-1}}\ell(n)^{\frac{r^{\nu}-r}{r-1}}p^{r^{\nu-1}}=: E(n,\nu).
\end{align*}
Taking the logarithm and using $p\ge \omega (\log n)^{-\frac{r}{r-1}}$ gives
\[
\log E(n,\nu) \ge r^{\nu-1}\log\omega-\frac{r}{r-1}\log^{(2)}n -\frac{r^{\nu}}{r-1}\log\nu-O(r^{\nu}).
\]
We will get a contradiction when we show that $\log E(n,\nu)\to\infty$. 
Note that
\[
r^{\nu-1}\log \omega -\frac{r^{\nu}}{r-1}\log\nu-O(r^{\nu})\ge \frac{r^{\nu}}{r-1}\log^{(2)}\omega-O(r^{\nu})\ge 0.5 \frac{r^{\nu}}{r-1}\log^{(2)}\omega,
\]
which gives
\[
\log E(n,\nu)\ge\frac{r\log^{(2)}n}{r-1}\left[0.5\,\frac{r^{\nu-1}\log^{(2)}\omega}{\log^{(2)}n}-1\right]
\]
and
\begin{align*}
\log\frac{r^{\nu-1}\log^{(2)}\omega}{\log^{(2)}n}&=(\nu-1)\log r+\log^{(3)}\omega-\log^{(3)}n\\
&\ge 0.99\,\omega^{\frac{r-1}{r}}\frac{\log r}{\log \omega}-\log^{(3)}n\to\infty,
\end{align*}
since $\omega\ge (3\log^{(3)}n\cdot\log^{(4)}n)^{\frac{r}{r-1}}$. 
Thus we have
\[
\liminf_{n\to\infty}\max\{p_j\ell(n):\,j\in [2,\nu]\}>0,
\]
implying  $\lim_{n\to\infty}p_0=1$. We emphasize that $p_0$ does not depend
on $m$. Choosing $m=m(n)\to\infty$ so slowly that 
$m(1-p_0)\to 0$, we conclude: with probability $1-o(1)$ every vertex in $[m]$ gets eventually infected. 
Using Lemma \ref{first m spreads} we complete the proof of the theorem.
\end{proof}

\subsection{Lower bound for the threshold}

Hoping that the upper bound for the percolation threshold is qualitatively sharp, we embark on a proof that whp there is no complete percolation for $p\le \omega^{-1}(\log n)^{-r/(r-1)}$, where $\omega\to\infty$ however slowly as $n\to\infty$. To this end, we introduce and analyze certain types of rooted graphs contained in $G_{n,k}$, whose presence is necessary for their roots to get eventually infected.

Let $B_0=A_0$ and $B_i=A_i\setminus A_{i-1}$ as before for $i\ge 1$.
For $j\ge 1$, if a vertex $x$ belongs to $B_j$, there must exist a subgraph of $G_{n,k}$ together with some initial infection conditions that certifies $x\in B_j$.  For instance, if $x\in B_1$, then $x$ must have $r$ neighbors in $B_0$. Hence $G_{n,k}$ contains a star with $r$-leaves and the central vertex $x$, where each of the $r$ leaves belongs to $B_0$. We will call such subgraphs witness graphs. 

\begin{definition}[Witness graphs]
 A \emph{witness graph of depth $j$ for $x$} is a rooted subgraph $W_j=W_j(x)$ of $G_{n,k}$ satisfying the following:
\begin{enumerate}[noitemsep, topsep=0pt]
\item[(I)] The vertex set of $W_j$ is $V=\cup_{i=0}^j V_i$, where
\begin{itemize}[noitemsep, topsep=0pt]
\item $V_i\subset B_i$ for all $i$
\item $V_j=\{x\}$
\end{itemize}
\item[(II)] For $i\ge 1$, each vertex in $V_i$ has at least one neighbor in $V_{i-1}$ and exactly $r$ neighbors in $V_0\cup \cdots \cup V_{i-1}$.
\end{enumerate}
If there are two adjacent vertices $u$ and $v$ in $W_j$, where $u\in V_s$ and $v\in V_t$ with $s<t$, then we say that $v$ is a parent of $u$ and we also say that $u$ is a child of $v$. 
Hence, except the vertices in $V_0$, every vertex has $r$ children, and except the root vertex, every vertex has at least one parent (there may be many). Finally,  although it is just a technicality, we view $W_j$ as a directed graph, where the edges are oriented toward the children. (See Figure~\ref{fig:WG}.)
\end{definition}

\begin{note}
The orientations of the edges in witness graphs are not necessarily the same as the orientations of the edges in the directed version of $G_{n,k}$, where all the edges are oriented toward smaller vertices. 
\end{note}

Note that witness graphs depend on $A_0$, the set of initially infected vertices, as well as $G_{n,k}$.
They cannot have directed cycles and they must  have at least $r$ vertices with out-degree 0. From now on, we will call the vertices in such graphs with out-degree 0 \textit{leaves} and the rest of the vertices \textit{internal vertices}. Leaves lie in $A_0$ and they spread the infection to other vertices of the witness graphs.

\begin{figure}[h]
\centering
\begin{tikzpicture}
 [scale=1,auto=left,every node/.style={ inner sep=2pt}, every loop/.style={min distance=20mm}]
 
 \tikzset{->-/.style={decoration={
  markings,
  mark=at position .5 with {\arrow{>}}},postaction={decorate}}}
  
 \fill[fill=blue!20!] (-3.5,-5.5) rectangle (6,-4.5);
 \node at (5.7,-5) {$V_0$};
 
 \fill[fill=orange!20!] (-3.5,-4.5) rectangle (6,-3.5);
 \node at (5.7,-4) {$V_1$};
 
  \fill[fill=red!20!] (-3.5,-3.5) rectangle (6,-2.5);
 \node at (5.7,-3) {$V_2$};
 
\fill[fill=yellow!20!] (-3.5,-2.5) rectangle (6,-1.5);
 \node at (5.7,-2) {$V_3$};
 
 \fill[fill=green!20!] (-3.5,-1.5) rectangle (6,-0.5);
 \node at (5.7,-1) {$V_4$};
 
  \fill[fill=brown!20!] (-3.5,-0.5) rectangle (6,0.5);
 \node at (5.7,0) {$V_5$};
 
\coordinate (n0) at (0,0) {};
\fill (n0) circle[radius=3pt] node[right] {$\ x$};

\coordinate (n1) at (1,-1) {};
\fill (n1) circle[radius=3pt] node[above] {$$};

\coordinate (n2) at (-1,-2) {};
\fill (n2) circle[radius=3pt] node {$$};

\coordinate (n3) at (2,-2) {};
\fill (n3) circle[radius=3pt] node {$$};

\coordinate (n4) at (0,-3) {};
\fill (n4) circle[radius=3pt] node[above] {$$};

\coordinate (n5) at (-2,-4) {};
\fill (n5) circle[radius=3pt] node[above] {$$};

\coordinate (n6) at (1,-4) {};
\fill (n6) circle[radius=3pt] node[above] {$$};

\coordinate (n7) at (3,-4) {};
\fill (n7) circle[radius=3pt] node[above] {$$};

\coordinate (n8) at (-3,-5) {};
\fill (n8) circle[radius=3pt] node[above] {$$};

\coordinate (n9) at (-1,-5) {};
\fill (n9) circle[radius=3pt] node[above] {$$};


\coordinate (n11) at (0,-5) {};
\fill (n11) circle[radius=3pt] node[above] {$$};

\coordinate (n12) at (2,-5) {};
\fill (n12) circle[radius=3pt] node[above] {$$};

\coordinate (n13) at (4,-5) {};
\fill (n13) circle[radius=3pt] node[above] {$$};

\foreach \from/\to in {n0/n1, n0/n2,n1/n2,n1/n3,n2/n4,n2/n5, n3/n7,n3/n4,n5/n8,n5/n9,n4/n6,n4/n9,n6/n12,n6/n11,n7/n13,n7/n12}
  \draw[thick,->-] (\from) -- (\to);
  
\end{tikzpicture}
\caption{A witness graph of depth 5 for $x$. Here $r=2$.}
\label{fig:WG}
\end{figure}
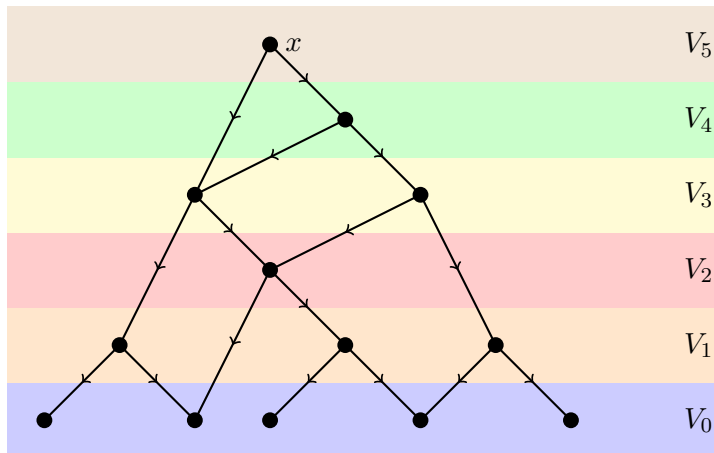

In order to show that a vertex is not infected whp, it is enough to show that it is not infected initially whp and the expected number of witness graphs for that vertex tends to 0.
In the rest of this section, we use $t$ and $\ell$ for the number of internal vertices and leaves, respectively. 
Recalling that $p\le \omega^{-1}(\log n)^{-r/(r-1)}$, we introduce $n_0=\omega^{1/2} (\log n)^{r/(r-1)}$. While $1/2$ could be replaced
with any number from $(0,1)$, the power of $\log n$ is the largest possible for $pn_0\to 0$. Under this condition, whp no vertex
in $[n_0]$ is initially infected, i.e. whp $[n_0]\cap A_0=\emptyset$. Using the witness graphs, we are going to show that whp no vertex from $[n_0]$ gets eventually infected either.

Let $V:=\{x^*, x_1,\dots, x_{\ell+t-1}\}$ stand for the vertex set of a generic witness graph.
We first consider the case when $\min_{j\ge 1} x_j\ge n_0$.

Let $L$ and $T\ni x$ be a partition of $V$, such that all vertices in $L$ have out-degree zero, all vertices in $T$ have out-degree $r$, and $\ell=|L|$ and $t=|T|$.
 Let $\outdx$ and $\indx$ denote, respectively, the out-degree and the in-degree of a vertex $x\in V$, and $\bold d:=\{\outdx,\indx\}_{x\in V}$.  So 
\begin{equation}\label{admissable d}
\outdx=
\begin{cases}
r, & \text{ if }x\in T\\
0, & \text{ if } x\in L
\end{cases}
\qquad \text{and} \qquad 
\begin{cases}
\indx>0, & \text{ if }x\neq x^*\\
\indx=0, & \text{ if } x=x^*.
\end{cases}
\end{equation}
The total number of (directed) edges is $m=rt=\sum_x \outdx =\sum_x \indx$. The total number 
of the digraphs on vertex set $V$, with the partition $V=T\cup L$, and the given in/out degrees is at most 
\[
\frac{m!}{\prod_{x\in V}\outdx!\,\indx!}.
\]
The event ``$(x',x)$ is an edge'' happens with probability $1-(1-1/[\max(x,x')-1])^k$, which is at most $K/\max(x,x')$, $K=\max(k,3)$,
see \eqref{<k/y}.  If $x$ is a child of $x'$  we will use the bound $K/x$. Furthermore, all the edge-indicators are negatively correlated. 
So the expected number of digraphs rooted at $x^*$, with an admissible  (as in~\eqref{admissable d}) in/out degree sequence, such that all vertices in $L$ are initially  infected, is at most
\begin{align}
&p^{\ell} \frac{(rt)!\, K^{rt}}{(r!)^t}\sum_{\bold{d} \text{ meets } \eqref{admissable d}}
\,\prod_{x}\frac{x^{-\indx}}{\indx!}\notag\\
&\qquad=p^{\ell} \frac{(rt)!\, K^{rt}}{(r!)^t}\prod_{x\neq x^*} x^{-1}\sum_{\bold d \text{ meets } \eqref{admissable d}}\, \prod_{x\neq x^*}
 \frac{1}{x^{\indx-1}\indx!}. \label{1bound}
\end{align} 
Here, since $x\ge n_0$ and 
\[
\sum_{x\neq x^*}(\indx-1)=rt -(t+\ell-1)=(r-1)t-\ell+1,
\]
the sum on the RHS of \eqref{1bound} is bounded above by
\begin{align}
n_0^{-(r-1)t+\ell-1}\!\!\!\!\!\!\!\sum_{\bold{d} \text{ meets } \eqref{admissable d}}\prod_{x\neq x^*}\frac{1}{\indx!}
&\le n_0^{-(r-1)t+\ell-1}\!\!\!\!\!\!\!\sum_{d_1,\dots,d_{t+\ell-1}\ge 0\atop
d_1+\cdots+d_{t+\ell-1}=rt}\,\,\prod_{j=1}^{t+\ell-1}\frac{1}{d_j!}\notag\\
&=n_0^{-(r-1)t+\ell-1}\,\frac{(t+\ell-1)^{rt}}{(rt)!},\label{2bound}
\end{align}
which depends on $V$ only through $t=|T|$ and $\ell=|L|$. The product $\prod_{x\neq x^*} x^{-1}$ also does not depend on the choice
of partition of $V$ into $T\setminus\{x^*\}$ and $L$. So we replace the sum in \eqref{1bound} with the RHS in \eqref{2bound}
and then sum the resulting bounds over all $V$, with $V=T\cup L$, $T\ni x^*$, $|T|=t$, $|L|=\ell$.  Observe that
\[
\sum_{V}\,\,\prod_{x\in V\setminus\{x^*\}}\frac{1}{x} \le  \frac{\binom{t+\ell-1}{\ell}}{(t+\ell-1)!}\left(\sum_{x=n_0}^n \frac{1}{x}\right)^{t+\ell-1} \le \frac{\binom{t+\ell-1}{\ell}}{(t+\ell-1)!}\cdot(\log n)^{t+\ell-1} .
\] 
Therefore the expected number of digraphs in question is bounded above by
\beq\label{expected digraphs}
p^{\ell}\,\frac{n_0^{-(r-1)t+\ell-1}k^{rt}}{(r!)^t}\cdot\frac{(t+\ell-1)^{rt}}{(t+\ell-1)!}\cdot\binom{t+\ell-1}{\ell}(\log n)^{t+\ell-1}.
\eeq
Using $\mu!\ge (\mu/e)^{\mu}$ and $t+\ell-1\le rt$ gives
\begin{align*}
\frac{(t+\ell-1)^{rt}}{(t+\ell-1)!} \le e^{t+\ell-1}(t+\ell-1)^{(r-1)t-\ell+1} \le (er)^{rt}t^{(r-1)t-\ell+1}.
\end{align*}
Combining this with $r!>(r/e)^r$ and $ \binom{t+\ell-1}{\ell}\le 2^{t+\ell-1}\le 2^{rt}$ in~\eqref{expected digraphs}, we get the following lemma.

\begin{lemma}\label{fn(t,ell)} Let $E_{t,\ell}(x^*)$ stand for the expected number of the witness graphs rooted at $x^*$, with
parameters $t$, $\ell$, such that the smallest non-root vertex in the graph is at least $n_0$. Then
\begin{align*}
&E_{t,\ell}(x^*)\le_b \mathcal E_{t,\ell}:=p^{\ell}(\log n)^{t+\ell-1}\left(\frac{t}{n_0}\right)^{(r-1)t-\ell+1}\!\!\! \lp 2Ke^2\rp^{rt}.
\end{align*}
\end{lemma}

We are now ready to prove Theorem~\ref{thm:percolation2}.

\begin{proof}[{\bf Proof of Theorem~\ref{thm:percolation2}}]

\noindent Let $t_0=\log n$ and let ${\cal W}_0$ denote the set of witness graphs with at most $t_0$ internal vertices whose roots lie in $[n_0]$ and all non-root vertices lie outside of $[n_0]$. Also let ${\cal W}_1$ denote the set of witness graphs with more than $t_0$ internal vertices.
Using Lemma~\ref{fn(t,ell)}, we will show that the following two events hold whp:
\begin{enumerate}[noitemsep, topsep=0pt]
\item[(i)] $\{{\cal W}_0=\emptyset\}$,
\item[(ii)] $\{{\cal W}_1=\emptyset\}$.
\end{enumerate}
Note that, together with the fact that $[n_0]\cap A_0=\emptyset$ whp, these two items imply that vertices in $[n_0]$ never get  infected whp. 
Let us show (i) first. By Lemma~\ref{fn(t,ell)},
\begin{align*}
\mean\big[|{\cal W}_0|\big] 
&\le n_0\sum_{t=1}^{t_0}\sum_{\ell=r}^{(r-1)t+1}E_{t,\ell}(x^*) \\
&\le \sum_{t=1}^{t_0}\, \ \sum_{\ell=r}^{(r-1)t+1}n_0\, p^{\ell}(\log n)^{t+\ell-1}\left(\frac{t}{n_0}\right)^{(r-1)t-\ell+1}\!\!\!\cdot \lp 2Ke^2\rp^{rt}
\end{align*}
Using $p\le \omega^{-1}(\log n)^{-r/(r-1)}$ and 
$n_0=\omega^{1/2}(\log n)^{r/(r-1)}$ and 
$t\le \log n$ on the right side above gives
\begin{align} \label{simplified expectation}
\mean\big[|{\cal W}_0|\big]  
&\le \sum_{t=1}^{t_0}\sum_{\ell=r}^{(r-1)t+1} \lp 2Ke^2\rp^{rt} \lp \frac{t}{\log n}\rp^{(r-1)t-\ell+1} \, \lp \omega^{-1/2}\rp^{(r-1)t+\ell}		\notag	\\
&\le \sum_{t=1}^{t_0}\sum_{\ell=r}^{(r-1)t+1} \lp 2Ke^2\rp^{rt} \lp \omega^{-1/2}\rp^{(r-1)t+\ell} \to 0,
\end{align}
which gives (i) above, i.e.\ whp no vertex in $[n_0]$ is infected through a witness graph with at most $t_0$ internal vertices.

In order to prove (ii), we first prove that whp ${\cal W}_1'=\emptyset$, 
where ${\cal W}_1'$ denotes the set of witness graphs with at least $t_0/r$ and at most $t_0$ internal vertices. 
Since the event (i) holds whp, it suffices to consider  the graphs in ${\cal W}_1'$ whose whole vertex set belongs to the interval $[n_0+1,n]$.
Using Lemma~\ref{fn(t,ell)} and similar computations that led to~\eqref{simplified expectation}, we see that the expected number of those special witness graphs in ${\cal W}_1'$ is at most
\begin{align*}
&n\sum_{t=t_0/r}^{t_0}\sum_{\ell=r}^{(r-1)t+1}{\cal E}_{t,\ell} \\
&\le \frac{n}{(\log n)^{r/(r-1)}} \sum_{t=t_0/r}^{t_0}\sum_{\ell=r}^{(r-1)t+1}  \lp 2Ke^2\rp^{rt} 
\lp \frac{t}{\log n}\rp^{(r-1)t-\ell+1} \lp \omega^{-1/2}\rp^{(r-1)t+\ell+1}		\\
&\le n \sum_{t=t_0/r}^{t_0}  \lp \lp 2Ke^2\rp^{r} \omega^{-(r-1)/2}\rp^{t} .
\end{align*}
Since $\omega\to \infty$ and $t \in [(\log n)/r,\log n]$ in the last sum above, the bound  tends to 0. Hence whp ${\cal W}_1'=\emptyset$. Assuming this event, we prove that ${\cal W}_1=\emptyset$, whence whp there are no witness graphs with more than $t_0$  vertices. 

Indeed, consider  hypothetical witness graphs in ${\cal W}_1$ and let $H$ be such a witness graph with the smallest number of internal vertices $s$, where $s>t_0$.
Let $x$ be the root of $H$. For every one of $r$ children $x'$ of $x$, the  subgraph $H'$ of $H$ rooted at $x'$ is either a witness graph itself or a single-vertex graph.  In the first case, by the minimality of $s$,  the number of internal vertices in $H'$ must be at most $t_0$. Furthermore, the total number of the internal vertices of these subgraphs is at least $s-1\ge t_0$. So there is
a subgraph $H'$ with at least $t_0/r$ internal vertices. This means that there is a witness graph $H'$ with the number of internal vertices lying in $[t_0/r,t_0]$, which 
contradicts our assumption that the event ${\cal W}_1'=\emptyset$ holds. Hence whp ${\cal W}_1$ is empty, which finishes the proof.
\end{proof}

\end{document}